\documentclass[12pt]{amsart}
\usepackage{amsfonts,graphicx,amssymb,amscd,amsmath,enumerate,verbatim,calc}

\def\NZQ{\mathbb}               

\def\ZZ{{\NZQ Z}}
\def\RR{{\NZQ R}}

\def\eb{{\bold e}}

\def\xb{{\bold x}}

\def\eb{{\bold e}}
\def\opn#1#2{\def#1{\operatorname{#2}}} 
\opn\aff{aff} \opn\con{conv} \opn\relint{relint} \opn\st{st}
\opn\lk{lk} \opn\cn{cn} \opn\core{core} \opn\vol{vol}
\opn\link{link} \opn\ini{in} \opn\int{int} \opn\aff{aff}

\def\Hc{{\mathcal H}}
\def\Ec{{\mathcal E}}

\def\Pc{{\mathcal P}}

\def\Ec{{\mathcal E}}
\def\Qc{{\mathcal Q}}
\newtheorem{Theorem}{Theorem}[section]
\newtheorem{Lemma}[Theorem]{Lemma}
\newtheorem{Corollary}[Theorem]{Corollary}
\newtheorem{Proposition}[Theorem]{Proposition}
\newtheorem{Question}[Theorem]{Question}
\theoremstyle{definition}
\newtheorem{Remark}[Theorem]{Remark}

\newtheorem{Example}[Theorem]{Example}
\newtheorem{Examples}[Theorem]{Examples}

\newtheorem{Problem}[Theorem]{Problem}

\let\epsilon\varepsilon
\let\phi=\varphi
\let\kappa=\varkappa
\numberwithin{equation}{section}

\begin{document}
\title{Minkowski sum of polytopes and its normality}
\author{Akihiro Higashitani}
\thanks{
\, \, \, The author is partially supported by JSPS Research Fellowship for Young Scientists.}

\address{Akihiro Higashitani,
Department of Mathematics, Graduate School of Science, 
Kyoto University, Kitashirakawa-Oiwake cho, Sakyo-ku, Kyoto, 606-8502, Japan}
\email{ahigashi@math.kyoto-u.ac.jp}

\maketitle

\begin{abstract}
In this paper, we consider the normality or the integer decomposition property (IDP, for short) 
for Minkowski sums of integral convex polytopes. 
We discuss some properties on the toric rings associated with Minkowski sums of integral convex polytopes. 
We also study Minkowski sums of edge polytopes and 
give a sufficient condition for Minkowski sums of edge polytopes to have IDP. 
\end{abstract}

\section{Introduction}

Normality and integer decomposition property 
are quite important properties on not only integral convex polytopes 
but also polytopal affine semigroup rings and toric varieties (consult, e.g., \cite{BrGu, BGT, CLS}). 

First of all, let us recall some definitions related to integral convex polytopes. 
Let $\Pc \subset \RR^N$ be an {\em integral} convex polytope, 
which is a convex polytope all of whose vertices have integer coordinates. 
Let $L(\Pc) \subset \ZZ^N$ be the affine lattice generated by the elements of $\Pc \cap \ZZ^N$, i.e., 
$L(\Pc)=\{v_0+\sum_{v \in \Pc \cap \ZZ^N} z_v (v-v_0) : z_v \in \ZZ \}$, where $v_0$ is some vertex of $\Pc$. 
\begin{itemize}
\item We say that $\Pc$ is {\em normal} if for any integer $k = 1, 2, \ldots$ and $\alpha \in k \Pc \cap L(k\Pc)$, 
where $k\Pc=\{k \alpha : \alpha \in \Pc\}$, 
there exist $\alpha_1, \ldots, \alpha_k$ belonging to $\Pc \cap \ZZ^N$ 
such that $\alpha = \alpha_1 + \cdots + \alpha_k$. 
\item We say that $\Pc$ has the {\em integer decomposition property} (IDP, for short) 
if for any integer $k = 1, 2, \ldots$ and $\alpha \in k \Pc \cap \ZZ^N$, 
there exist $\alpha_1, \ldots, \alpha_k$ belonging to $\Pc \cap \ZZ^N$ 
such that $\alpha = \alpha_1 + \cdots + \alpha_k$. Thus, if $\Pc$ has IDP, then $\Pc$ is normal. 
What $\Pc$ has IDP is also called what $\Pc$ is {\em integrally closed}. 
It is well-known that $\Pc$ always has IDP when $\dim \Pc \leq 2$. 
\item For some subsets $A_1,\ldots,A_m$ of $\RR^N$, 
let $A_1+\cdots+A_m=\{\sum_{i=1}^ma_i : a_i \in A_i, 1 \leq i \leq m\}$. 
This set $A_1+\cdots+A_m$ is called the {\em Minkowski sum} of $A_1,\ldots,A_m$. 
Note that when $\Pc \subset \RR^N$ is a convex set, 
the Minkowski sum $\underbrace{\Pc+\cdots+\Pc}_m$ of $m$ copies of $\Pc$ coincides with 
the dilation $m\Pc$, where $m\Pc=\{m \alpha : \alpha \in \Pc\}$. 
Hence, taking Minkowski sum of convex polytopes can be regarded 
as a generalization of a dilation of a convex polytope. 
\end{itemize}

Let $K$ be a field and $R$ the $K$-algebra $K[\xb^{\pm},t]=K[x_1,x_1^{-1}, \ldots, x_N,x_N^{-1},t]$. 
For $\alpha=(\alpha_1,\ldots,\alpha_N) \in \ZZ^N$, let $\xb^\alpha$ denote 
the Laurent monomial $x_1^{\alpha_1} \cdots x_N^{\alpha_N} \in R$. 
Given an integral convex polytope $\Pc \subset \RR^N$, we define two $K$-algebras $K[\Pc]$ and $\Ec_K(\Pc)$ as follows. 
\begin{enumerate}
\item Let $K[\Pc] \subset R$ be the $K$-algebra generated by $\{\xb^\alpha t : \alpha \in \Pc \cap \ZZ^N\}$, 
that is, $$K[\Pc]=K[\xb^\alpha t : \alpha \in \Pc \cap \ZZ^N].$$ 
\item Let $\Ec_K(\Pc) \subset R$ be the $K$-algebra defined by 
\begin{align*}
\Ec_K(\Pc)=K[\xb^\alpha t^n : \alpha \in n \Pc \cap \ZZ^N, n \in \ZZ_{\geq 0}]. 
\end{align*}
\end{enumerate}
These algebras are finitely generated graded $K$-algebras, 
where their grading is defined by $\deg(\xb^\alpha t^n)=n$ for $\alpha \in n \Pc \cap \ZZ^N$. 
We call $K[\Pc]$ the {\em toric ring} of $\Pc$ and $\Ec_K(\Pc)$ the {\em Ehrhart ring} of $\Pc$. 
Notice that $\Pc$ is normal if and only if so is $K[\Pc]$, and 
$\Pc$ has IDP if and only if $K[\Pc]=\Ec_K(\Pc)$.


\bigskip

In the outstanding paper \cite{BGT}, Bruns, Gubeladze and Trung proved the following: 
\begin{Theorem}[{\cite[Theorem 1.3.3]{BGT}}]\label{motivation}
Let $\Pc$ be an integral convex polytope $\Pc$ of dimension $d$. Then the following hold: 
\begin{itemize}
\item[(a)] $K[n\Pc]$ is normal if $n \geq d-1$; 
\item[(b)] $K[n\Pc]$ is Koszul if $n \geq d$; 
\item[(c)] $K[n\Pc]$ is level of a-invariant $-1$ if $n \geq d+1$. 
\end{itemize}
\end{Theorem}

This theorem conserns the ``dilation'' of polytope. 
Hence it is natural to think of whether we can extend those results to the ``Minkowski sum of polytopes''. 
In this paper, we extend Theorem \ref{motivation} (a) and (c) to the Minkowski sum of polytopes with a slight modification. 
More precisely, we prove that for integral convex polytopes $\Pc_1,\ldots,\Pc_m \subset \RR^d$ 
of dimension $d_1,\ldots,d_m$, respectively, $K[n_1\Pc_1+\cdots+n_m\Pc_m]$ is normal 
if $n_i \geq d_i$ for each $1 \leq i \leq m$ and is level of a-invariant $-1$ 
if $n_i \geq d_i+1$ for each $1 \leq i \leq m$ (Theorem \ref{M_IDP}).

\bigskip

We are also, in particular, interested in the following: 
\begin{Problem}\label{toi}
For integral convex polytopes $\Pc_1,\ldots,\Pc_m \subset \RR^N$, 
when is $\Pc_1+\cdots + \Pc_m$ normal? Or, when does $\Pc_1+\cdots + \Pc_m$ have IDP?
\end{Problem}
Of course, Theorem \ref{M_IDP} (a) gives one solution. 
However, except for this, little result is known on normality or IDP for Minkowski sums of integral convex polytopes. 
Moreover, as is shown by the following example, 
Minkowski sums of integral convex polytopes having IDP is not necessarily normal. 
\begin{Example}[{See \cite[p.2315]{workshop}}] 
Let $\Pc_1=\con(\{(0,0,0),(1,0,0),(0,1,0)\}) \subset \RR^3$ and $\Pc_2=\con(\{(0,0,0),(1,1,3)\}) \subset \RR^3$. 
Since each of $\dim \Pc_1$ and $\dim \Pc_2$ is at most 2, each of $\Pc_1$ and $\Pc_2$ has IDP. 
However, we see that $\Pc_1+\Pc_2$ is not normal. In particular, this does not have IDP. 
\end{Example}

On the other hand, it is obvious from the definition that 
for an integral convex polytope $\Pc$ having IDP, $n\Pc$ always has IDP for every $n \in \ZZ_{>0}$. 
In addition, when we dilate an integral convex polytope of dimension $d$ 
by at least $(d-1)$ times, the dilated polytope always has IDP (see \cite[\S 2.2]{CLS}). 
In this paper, for the development of the study of IDP for Minkowski sums of integral convex polytopes, 
we investigate the Minkowski sum of integral convex polytopes arising from graphs, 
called {\em edge polytopes} (see Section 3). 
We give a sufficient condition for Minkowski sums of edge polytopes to have IDP (Theorem \ref{subgraph2}).

A brief organization of this paper is as follows. 
First, in Section 2, we prove an extended version of Theorem \ref{motivation} (a) and (c) 
(Theorem \ref{M_IDP}, which is the main result of this paper). 
Next, in Section 3, we study Minkowski sums of edge polytopes. 
After we recall some notions and definitions on graphs, we define the edge polytope  
and discuss the dimesnion of Minkowski sum of edge polytopes (Proposition \ref{jigen}) in Section 3.1. 
We also give a sufficient condition for Minkowski sums of edge polytopes to have IDP (Theorem \ref{subgraph2}) 
in Section 3.2. Finally, we give some examples concerning Theorem \ref{subgraph2} in Section 3.3. 

\bigskip

\section{Toric rings of Minkowski sums of polytopes}\label{generalization}

In this section, we extend Theorem \ref{motivation} (a) and (c) 
from ``dilations'' of integral convex polytopes to ``Minkowski sum''. 

Before it, we discuss the dimension of Minkowski sums of polytopes. 
\begin{Proposition}\label{jigena}
Let $\Pc \subset \RR^N$ and $\Pc' \subset \RR^N$ be convex polytopes. 
For any $\ell>0$, we have $\dim(\Pc+\Pc')=\dim(\ell \Pc + \Pc')$. 
\end{Proposition}
\begin{proof}
Let $d=\dim(\Pc+\Pc')$. Then there exist $(d+1)$ affinely independent 
vectors $v_0+v_0', v_1+v_1', \ldots, v_d+v_d'$ in $\Pc+\Pc'$, 
where $v_i \in \Pc$ and $v_i' \in \Pc'$ for each $0 \leq i \leq d$. 
Then $\ell v_0+v_0', v_1+(\ell-1) v_0+v_1', \ldots, v_d+(\ell-1)v_0+v_d'$ are affinely independent 
and each of them belongs to $\ell \Pc + \Pc'$. 
Hence, we obtain $\dim (\Pc + \Pc') \leq \dim (\ell \Pc + \Pc').$ 
On the other hand, let $\Qc=\ell \Pc$ and $\ell'=1/\ell$. By the above discussion, we also obtain that 
$\dim(\ell \Pc + \Pc')= \dim (\Qc + \Pc') \leq \dim ( \ell' \Qc + \Pc')=\dim (\Pc + \Pc').$
\end{proof}

As a corollary of this proposition, we also see: 
\begin{Corollary}
Let $\Pc_1,\ldots,\Pc_m \subset \RR^N$ be convex polytopes and let $n_i>0$ for each $1 \leq i \leq m$. 
Then $$\dim(\Pc_1+\cdots+\Pc_m)=\dim(n_1\Pc_1+\cdots+n_m\Pc_m).$$
\end{Corollary}
\begin{proof}
We prove the assertion by induction on $m$. This is obvious when $m=1$. 
When $m=2$, by Proposition \ref{jigena}, we have 
$\dim (\Pc_1 + \Pc_2) = \dim( n_1 \Pc_1+ \Pc_2)=\dim(n_1\Pc+n_2\Pc_2)$. 
When $m > 2$, by Proposition \ref{jigen} together with the inductive hypothesis, we obtain 
\begin{align*}
\dim(\Pc_1+\cdots+\Pc_m)&=\dim(n_2(\Pc_1+\Pc_2)+n_3\Pc_3+\cdots+n_m\Pc_m) \\
&=\dim\left(\frac{n_1}{n_2} n_2\Pc_1+(n_2\Pc_2+\cdots+n_m\Pc_m)\right) \\
&= \dim(n_1\Pc_1+\cdots+n_m\Pc_m). 
\end{align*}
\end{proof}


\bigskip

Now we prove the main theorem of this paper. 
\begin{Theorem}\label{M_IDP}
Let $\Pc_1,\ldots,\Pc_m \subset \RR^N$ be integral convex polytopes 
and let $d_i=\dim \Pc_i$ for $1 \leq i \leq m$. Given positive integers $n_1,\ldots,n_m$, 
the following hold: 
\begin{itemize}
\item[(a)]
$n_1\Pc_1+\cdots+n_m\Pc_m$ has IDP (in particular, $K[n_1\Pc_1+\cdots+n_m\Pc_m]$ is normal) 
if $n_i \geq d_i$ for each $1 \leq i \leq m$; 
\item[(b)]
$K[n_1\Pc_1+\cdots+n_m\Pc_m]$ is level of a-invariant $-1$ if $n_i \geq d_i+1$ for each $1 \leq i \leq m$. 
\end{itemize}
\end{Theorem}

For the proof of this theorem, we prove Lemma \ref{naibu} and Lemma \ref{kagi}. 

For $A \subset \RR^N$, let $\aff(A)$ be the affine subspace of $\RR^N$ spanned by $A$. 
We denote $\int(A)$ by the relative interior of $A$ with respect to $\aff(A)$. 
\begin{Lemma}\label{naibu}
Let $\Pc_1,\ldots,\Pc_m \subset \RR^N$ be convex polytopes. Then one has 
$$\int(\Pc_1+\cdots+\Pc_m)=\int(\Pc_1)+\cdots+\int(\Pc_m).$$ 
\end{Lemma}
\begin{proof}
It suffices to show the case $m=2$. Namely, our goal is to prove that 
for convex polytopes $\Pc$ and $\Qc$ in $\RR^N$, we have $\int(\Pc)+\int(\Qc)=\int(\Pc+\Qc)$. 

Take $x \in \int(\Pc)$ and $y \in \int(\Qc)$. 
Then there is an open subset $U \subset \Pc$ (resp. $V \subset \Qc$) 
with respect to $\aff(\Pc)$ (resp. $\aff(\Qc)$) such that $x \in U$ (resp. $y \in V$). 
Then $x+y \in U+V \subset \Pc + \Qc$. Moreover, $U+V$ is an open set with respect to 
$\aff(\Pc)+\aff(\Qc)=\aff(\Pc+\Qc)$. Hence, $x+y \in \int(\Pc+\Qc)$. 


On the other hand, let $z \in \int(\Pc+\Qc)$. 
Then there are $x$ and $y$ in $\Pc+\Qc$ such that $z=rx + (1-r)y$ for some $0 < r < 1$. 
Moreover, there are $x_1 \in \Pc$ and $x_2 \in \Qc$ (resp. $y_1 \in \Pc$ and $y_2 \in \Qc$) 
such that $x=x_1+x_2$ (resp. $y=y_1+y_2$). Hence, $z=(rx_1+(1-r)y_1)+(rx_2+(1-r)y_2)$. 
Since $rx_1+(1-r)y_1 \in \int(\Pc)$ and $rx_2+(1-r)y_2 \in \int(\Qc)$, we conclude that $z \in \int(\Pc)+\int(\Qc)$. 
\end{proof}

\begin{Lemma}\label{kagi}
Work with the same notation as in Theorem \ref{M_IDP}. 
\begin{itemize}
\item[(a)]
If $n_i \geq d_i+1$ for each $i$, then we have 
\begin{equation}\label{tousiki}
\begin{aligned}
&(n_1\Pc_1+\cdots+n_m\Pc_m) \cap \ZZ^N = \\
&\quad\quad\quad(d_1\Pc_1+\cdots+d_m\Pc_m ) \cap \ZZ^N+
\sum_{i=1}^m (\underbrace{\Pc_i \cap \ZZ^N + \cdots + \Pc_i \cap \ZZ^N}_{n_i-d_i}). 
\end{aligned}
\end{equation}
\item[(b)]If $n_i \geq d_i+2$ for each $i$, then we have 
\begin{equation}\label{tousiki2}
\begin{aligned}
&\int(n_1\Pc_1+\cdots+n_m\Pc_m) \cap \ZZ^N = \\
&\int((d_1+1)\Pc_1+\cdots+(d_m+1)\Pc_m ) \cap \ZZ^N+
\sum_{i=1}^m (\underbrace{\Pc_i \cap \ZZ^N + \cdots + \Pc_i \cap \ZZ^N}_{n_i-d_i-1}). 
\end{aligned}
\end{equation}
\end{itemize}
\end{Lemma}
\begin{proof}
(a) Let $\alpha \in (n_1\Pc_1+\cdots+n_m\Pc_m) \cap \ZZ^N$. Then 
there is $w_i \in n_i\Pc_i$ for each $i$ such that $\alpha=w_1+\cdots+w_m$. 
By Carath\'eodory's Theorem (cf. \cite[Corollary 7.1i]{Sch}), 
there are $(d_i+1)$ affinely independent vertices $v_0^{(i)}, v_1^{(i)},\ldots,v_{d_i}^{(i)} \in \Pc_i \cap \ZZ^N$ of $\Pc_i$ 
such that $w_i=\sum_{j=0}^{d_i} r_j^{(i)} v_j^{(i)}$, where $r_j^{(i)} \geq 0$ and $\sum_{j=0}^{d_i}r_j^{(i)}=n_i$. Thus, 
$\alpha$ can be written like $$\alpha=\sum_{j=0}^{d_1} r_j^{(1)} v_j^{(1)} + \cdots + \sum_{j=0}^{d_m} r_j^{(m)} v_j^{(m)}.$$ 
Since $n_i \geq d_i+1$ and $\sum_{j=0}^{d_i}r_j^{(i)}=n_i$ for each $i$, there is an index $k_i$ such that $r_{k_i}^{(i)} \geq 1$. 
Then $\alpha$ can be decomposed like $\alpha'+\beta^{(1)}+\cdots+\beta^{(m)}$, where 
$\alpha' \in ((n_1-1)\Pc_1+\cdots+(n_m-1)\Pc_m) \cap \ZZ^N$ and $\beta^{(i)} \in \Pc_i \cap \ZZ^N$. 
Since we can do this decomposition whenever $n_i \geq d_i+1$, 
we conclude that $\alpha$ belongs to the right-hand side of \eqref{tousiki}. 
This shows one inclusion. On the other hand, another inclusion is easy to see. 

\noindent
(b) Let $\alpha \in \int(n_1\Pc_1+\cdots+n_m\Pc_m) \cap \ZZ^N$. 
By Lemma \ref{naibu}, we have $\int(n_1\Pc_1+\cdots+n_m\Pc_m)=\int(n_1\Pc_1)+\cdots+\int(n_m\Pc_m)$. 
Thus, there is $w_i \in \int(n_i\Pc_i)$ for each $i$ such that $\alpha=w_1+\cdots+w_m$. 
Then there are $(d_i'+1)$ affinely independent vertices 
$v_0^{(i)}, v_1^{(i)},\ldots,v_{d_i'}^{(i)} \in \Pc_i \cap \ZZ^N$ of $\Pc_i$ 
such that $w_i=\sum_{j=0}^{d_i'} r_j^{(i)} v_j^{(i)}$, where $d_i' \leq d_i$, $r_j^{(i)} > 0$ and $\sum_{j=0}^{d_i'}r_j^{(i)}=n_i$. 
Thus, $\alpha$ can be written like $$\alpha=\sum_{j=0}^{d_1'} r_j^{(1)} v_j^{(1)} + \cdots + \sum_{j=0}^{d_m'} r_j^{(m)} v_j^{(m)}.$$ 
Since $n_i \geq d_i+2 \geq d_i'+2$ and $\sum_{j=0}^{d_i'}r_j^{(i)}=n_i$ for each $i$, there is an index $k_i$ such that $r_{k_i}^{(i)} > 1$. 
Then $\alpha$ can be decomposed like $\alpha'+\beta^{(1)}+\cdots+\beta^{(m)}$, where 
$\alpha' \in \int((n_1-1)\Pc_1+\cdots+(n_m-1)\Pc_m) \cap \ZZ^N$ and $\beta^{(i)} \in \Pc_i \cap \ZZ^N$. 
Since we can do this decomposition whenever $n_i \geq d_i+2$, 
we conclude that $\alpha$ belongs to the right-hand side of \eqref{tousiki2}. 
This shows one inclusion. Another inclusion also follows easily. 
\end{proof}

\smallskip

\begin{proof}[Proof of Theorem \ref{M_IDP}]
(a) Let $\alpha \in n(n_1\Pc_1 + \cdots + n_m \Pc_m) \cap \ZZ^N$ with $n \geq 2$. 
Then $\alpha \in (nn_1\Pc_1 + \cdots + nn_m \Pc_m) \cap \ZZ^N$ and $nn_i \geq d_i+1$ for each $i$. 
By Lemma \ref{kagi} (a), $\alpha$ can be written like $\alpha'+\sum_{i=1}^m\sum_{j=1}^{nn_i-d_i}\beta_j^{(i)}$, 
where $\alpha' \in (d_1\Pc_1 + \cdots + d_m \Pc_m) \cap \ZZ^N$ and $\beta_j^{(i)} \in \Pc_i \cap \ZZ^N$. 
Since $n_i \geq d_i$, it is easy to see that $\alpha'+\sum_{i=1}^m\sum_{j=1}^{nn_i-d_i}\beta_j^{(i)}$ 
can be decomposed into $n$ integer points belonging to $(n_1\Pc_1 + \cdots + n_m \Pc_m) \cap \ZZ^N$. 
This implies that $n_1\Pc_1 + \cdots + n_m \Pc_m$ has IDP. 

\noindent
(b) It is enough to show that for any $\alpha \in n\int(n_1\Pc_1 + \cdots + n_m \Pc_m) \cap \ZZ^N$ with $n \geq 2$, 
$\alpha$ can be written like $\alpha=\beta+\beta_1+\cdots+\beta_{n-1}$, 
where $\beta \in \int(n_1\Pc_1 + \cdots + n_m \Pc_m) \cap \ZZ^N$ and 
$\beta_1,\ldots,\beta_{n-1} \in (n_1\Pc_1 + \cdots + n_m \Pc_m) \cap \ZZ^N$. 

Given $n \geq 2$, let $\alpha \in n\int(n_1\Pc_1 + \cdots + n_m \Pc_m) \cap \ZZ^N$. 
Then we have $n\int(n_1\Pc_1 + \cdots + n_m \Pc_m)=\int(nn_1\Pc_1+\cdots+nn_m\Pc_m)$ by Lemma \ref{naibu}. 
Moreover, by Lemma \ref{kagi} (b), $\alpha$ can be expressed like 
$\alpha'+\sum_{i=1}^m\sum_{j=1}^{nn_i-d_i-1}\beta_j^{(i)}$, 
where $\alpha' \in \int((d_1+1)\Pc_1 + \cdots + (d_m+1) \Pc_m) \cap \ZZ^N$ and $\beta_j^{(i)} \in \Pc_i \cap \ZZ^N$. 
Let $\alpha''=\alpha'+\sum_{i=1}^m\sum_{j=1}^{n_i-d_i-1}\beta_j^{(i)}$. 
Then we have $\alpha'' \in \int(n_1\Pc_1 + \cdots + n_m \Pc_m) \cap \ZZ^N$ and $\alpha$ can be rewritten like 
$\alpha = \alpha''+\sum_{k=1}^{n-1}(\gamma_1^{(k)}+\cdots+\gamma_m^{(k)})$, 
where $\gamma_i^{(k)} \in n_i \Pc_i \cap \ZZ^N$. 
%
%
%
%
%
%
%
\end{proof}

\begin{Remark}
For Theorem \ref{M_IDP} (a), if there is $i$ such that $n_i=d_i-1$, 
then Theorem \ref{M_IDP} (a) is no longer true. For example, let 
$\Pc_1=\con(v_1,v_2,v_3)$ and  $\Pc_2=\con(v_4,v_5,v_6)$, where 
\begin{align*}
&v_1=(1,1,0,0,0,0), \; v_2=(0,0,1,1,0,0), \; v_3=(0,0,0,0,1,1), \\
&v_4=(1,0,0,0,0,1), \; v_5=(0,1,1,0,0,0) \text{ and } v_6=(0,0,0,1,1,0). 
\end{align*}
Then each of $\Pc_1,\Pc_2 \subset \RR^6$ is of dimension 2. 
We consider $n\Pc_1+\Pc_2$, where $n \geq 1$. Then we have 
\begin{align*}
\left(\frac{6n-2}{3}v_1+\frac{1}{3}v_2+\frac{1}{3}v_3\right)
+\frac{2}{3}\left(v_4+v_5+v_6\right)
=(2n,2n,1,1,1,1) \in 2(n\Pc_1+\Pc_2) \cap \ZZ^6.
\end{align*}
Moreover, one sees that 
\begin{align*}
(n\Pc_1+\Pc_2)\cap \ZZ^6&=(\underbrace{\Pc_1 \cap \ZZ^N+\cdots + \Pc_1 \cap \ZZ^6}_n+\Pc_2 \cap \ZZ^6) \\
&\cup (\{(1,1,1,1,1,1)\}+\underbrace{\Pc_1 \cap \ZZ^N+\cdots + \Pc_1 \cap \ZZ^6}_{n-2})
\end{align*}
when $n \geq 2$. Hence, $(2n,2n,1,1,1,1) \in 2(n\Pc_1+\Pc_2)\cap \ZZ^6$ cannot be written as a sum of 
any two integer points contained in $(n\Pc_1+\Pc_2)\cap \ZZ^6$. Namely, 
$n\Pc_1+\Pc_2$ does not have IDP. 
\end{Remark}

For Theorem \ref{motivation} (b), we remain the following: 
\begin{Question}
Work with the same notation as in Theorem \ref{M_IDP}. 
Is it true that $K[n_1\Pc_1 + \cdots + n_m\Pc_m]$ is Koszul 
or the defining ideal (toric ideal) of $K[n_1\Pc_1 + \cdots + n_m\Pc_m]$ has a quadratic Gr\"obner basis 
if $n_i \geq d_i$ for every $1 \leq i \leq m$? 
\end{Question}

\section{Minkowski sum of edge polytopes}\label{minkowski_sum_edge}

In this section, we study IDP of Minkowski sums of edge polytopes. 
After fixing our notation on simple graphs, 
we define the edge polytope and study the dimension of Minkowski sum of edge polytopes (Proposition \ref{mjigen}). 
We also consider the problem when the Minkowski sum of edge polytopes has IDP 
(Theorem \ref{subgraph2}). Finally, we supply some examples of graphs 
which show that the conditions described in Theorem \ref{subgraph2} are necessary 
for Minkowski sums of edge polytopes to have IDP.

\subsection{Dimension of Minkowski sum of edge polytopes}\label{dim}
Let $G$ be a simple graph on the vertex set $V(G)$ with the edge set $E(G)$. 
Throughout this paper, we always assume that graphs are simple, so we omit to say ``simple''. 
We recall several terminologies on graphs. 
\begin{itemize}
\item A graph $G$ is called {\em bipartite} if $V(G)$ can be decomposed into two non-empty subsets 
$U$ and $V$ of $V(G)$ such that $V(G)=U \cup V$, $U \cap V = \emptyset$ and 
every edge $\{i,j\} \in E(G)$ belongs to $U \times V$. 
We also call this partition $U \cup V$ the {\em partition} of the bipartite graph $G$. 
\item A sequence $v_0,v_1,\ldots,v_k$ of vertices in $G$ is called a {\em walk} if 
$\{v_{i-1},v_i\} \in E(G)$ for each $1 \leq i \leq k$. 
A walk is called a {\em path} if $v_i$'s are all distinct. Moreover, 
a walk is called a {\em cycle} if $v_0,v_1,\ldots,v_{k-1}$ are distinct and $v_0=v_k$. 
\item The length of a walk (a cycle) $v_0,v_1,\ldots,v_k$ is defined by $k$. 
A walk in $G$ is called {\em odd} (resp. {\em even}) if its length is odd (resp. even). 
It is well known that $G$ is bipartite if and only if $G$ has no odd cycle. 
\item A subgraph of $G$ is called {\em spanning} if its vertex set is equal to that of $G$. 
\item A {\em forest} is a graph without any cycle. A {\em tree} is a connected forest. 
Note that every forest is bipartite. 
\item We say that $G$ is {\em 2-connected} if the induced subgraph 
with the vertex set $V(G) \setminus \{v\}$ is still connected for any vertex $v$ of $G$. 
A subgraph is called {\em 2-connected component} if it is a maximal 2-connected subgraph. 
\end{itemize}

Let $V(G)=\{1,\ldots,d\}$. 
For $1 \leq i \leq d$, let $\eb_i$ be the $i$th coordinate vectors of $\RR^d$. 
Given an edge $\{i,j\} \in E(G)$, let $\rho(e) \in \RR^d$ denote the vector $\eb_i+\eb_j$. 
We write $\Pc_G$ for the convex hull of the set of integer points $\{\rho(e) : e \in E(G)\}$. 
We call this polytope $\Pc_G$ the {\em edge polytope} of $G$.

In \cite{OH}, Ohsugi and Hibi studied some properties on edge polytopes. 
For example, they obtain the dimension of an edge polytope as follows. 
\begin{Proposition}[{\cite[Proposition 1.3]{OH}}]\label{jigen}
Let $G$ be a connected graph with $d$ vertices. Then one has 
\begin{align*}
\dim \Pc_G=
\begin{cases}
d-2, \;\;&\text{ if $G$ is bipartite}, \\
d-1, \;\;&\text{ if $G$ is non-bipartite}. 
\end{cases}
\end{align*}
\end{Proposition}

Similar to this proposition, 
we discuss the dimension of the Minkowski sum of some edge polytopes. 

Let $G_1,\ldots,G_m$ be graphs on the same vertex set $\{1,\ldots,d\}$. 
Let $E(G_i)$ be the edge set of $G_i$ for each $1 \leq i \leq m$. 
We denote by $G_1+\cdots+G_m$ 
the graph on the vertex set $\{1,\ldots,d\}$ with the edge set $\bigcup_{i=1}^mE(G_i)$. 
\begin{Proposition}\label{mjigen}
Let $G_1,\ldots,G_m$ be connected graphs on the same vertex set $\{1,\ldots,d\}$. 
Then one has 
\begin{align*}
\dim(\Pc_{G_1} + \cdots + \Pc_{G_m})=
\begin{cases}
d-2, \quad&\text{ if $G_1+\cdots+G_m$ is bipartite}, \\
d-1, \quad&\text{ if $G_1+\cdots+G_m$ is non-bipartite}. 
\end{cases}
\end{align*}
\end{Proposition}
\begin{proof}
Let $G=G_1+\cdots+G_m$ and let $\Pc=\Pc_{G_1}+\cdots + \Pc_{G_m}$. 
Since $\Pc$ is contained in the hyperplane defined by 
$\{(x_1,\ldots,x_d) \in \RR^d : \sum_{i=1}^d x_i = 2m\}$, one has $\dim \Pc \leq d-1$. 
On the other hand, since each $G_i$ is connected, each $G_i$ contains a spanning tree. 
Thus, in particular, $G_1$ contains $(d-1)$ edges $e_1,\ldots,e_{d-1}$ such that 
$\rho(e_1),\ldots,\rho(e_{d-1})$ are affinely independent. 
Thus $d-2 \leq \dim \Pc_{G_1} \leq \dim \Pc$. Hence, we have $d-2 \leq \dim \Pc \leq d-1$.

Assume that $G$ is bipartite. Let $U \cup V$ be the partition of $G$. 
Then we see that $\Pc$ is contained in the hyperplane defined by 
$\{(x_1,\ldots,x_d) \in \RR^d : \sum_{ i \in U} x_i = \sum_{j \in V} x_j\}$. 
This implies that $\dim \Pc \leq d-2$. Thus we obtain $\dim \Pc=d-2$. 

Assume that $G$ is not bipartite. 
If there is $1 \leq i \leq m$ such that $G_i$ is not bipartite, 
then $\dim \Pc \geq \dim \Pc_{G_i} = d-1$ by Proposition \ref{jigen}. Hence, $\dim \Pc=d-1$. 
If each $G_i$ is bipartite, since $G$ is non-bipartite, 
there exists an edge $f \in \bigcup_{i=2}^m E(G_i)$ such that $G_1 \cup \{f\}$ has an odd cycle. 
We assume that $f$ is an edge of $G_2$. Let $U_1 \cup V_1$ be the partition of $G_1$.
Then $f \not\in U_1 \times V_1$. Thus $f \in U_1 \times U_1$ or $f \in V_1 \times V_1$, 
say, $f \in U_1 \times U_1$. Since $G_2$ is connected, there is an edge $f' \in E(G_2)$ 
such that $f' \not\in U_1 \times U_1$. Fix some edges $f_i \in E(G_i)$ for each $3 \leq i \leq m$ 
and let $v=\sum_{i=3}^m\rho(f_i)$. Let $e_1,\ldots,e_{d-1}$ be edges in $G_1$ forming its spanning tree. 
We consider the $d$ integer points 
$$\rho(e_1)+\rho(f)+v,\rho(e_2)+\rho(f)+v,\ldots,\rho(e_{d-1})+\rho(f)+v,\rho(e_1)+\rho(f')+v,$$
where each of them belongs to $\Pc \cap \ZZ^d$. Let 
\begin{align*}
v_i=
\begin{cases}
(\rho(e_1)+\rho(f')+v)-(\rho(e_1)+\rho(f)+v)=\rho(f')-\rho(f), \;\;&i=1, \\
(\rho(e_i)+\rho(f)+v)-(\rho(e_1)+\rho(f)+v)=\rho(e_i)-\rho(e_1), &i=2,\ldots,d-1. 
\end{cases}
\end{align*}
If there is $(r_1,\ldots,r_{d-1}) \in \RR^{d-1}$ such that $\sum_{i=1}^{d-1}r_iv_i=0$, 
then $r_1(\rho(f)-\rho(f'))=\sum_{i=2}^{d-1}r_i(\rho(e_i)-\rho(e_1))$. 
Since each $\rho(e_i)-\rho(e_1)$ is contained in the hyperplane 
$\Hc = \{(x_1,\ldots,x_d) \in \RR^d : \sum_{i \in U_1}x_i=\sum_{j \in V_1}x_j = 0\} \subset \RR^d$, 
so should be $r_1(\rho(f)-\rho(f'))$. However, since $f \in U_1 \times U_1$ 
and $f' \not\in U_1 \times U_1$, $\rho(f)-\rho(f')$ is never contained in $\Hc$. 
Hence $r_1=0$. Moreover, since $\rho(e_1),\ldots,\rho(e_{d-1})$ are affinely independent, 
one has $r_2=\cdots=r_{d-1}=0$. Thus $r_1=r_2=\cdots=r_{d-1}=0$. 
This implies that $v_1,\ldots,v_{d-1}$ are linearly independent. 
Hence $\dim \Pc \geq d-1$. Therefore, it follows that $\dim \Pc=d-1$, as required. 
\end{proof}


\bigskip

\subsection{A sufficient condition for Minkowski sums of edge polytopes to have IDP}\label{when_IDP}

In this section, we discuss the problem when the Minkowski sum of edge polytopes has IDP. 
Namely, we give a partial answer for Problem \ref{toi} in the case of edge polytopes. 

We say that a connected graph $G$ satisfies the {\em odd cycle condition} 
if for arbitrary two odd cycles $C$ and $C'$ in $G$ which have no common vertex, 
there exists an edge of $G$ joining some vertex of $G$ with some vertex of $G'$. 
For the normality or IDP of edge polytopes, the following is known. 
\begin{Theorem}[\cite{OH}, see also \cite{SVV}]\label{occ}
Let $G$ be a connected graph. Then the following four conditions are equivalent: 
\begin{itemize}
\item[(a)] $\Pc_G$ is normal; 
\item[(b)] $\Pc_G$ has IDP; 
\item[(c)] $\Pc_G$ has a unimodular covering; 
\item[(d)] $G$ satisfies the odd cycle condition. 
\end{itemize}
\end{Theorem}
Note that although the equivalence of (a) and (b) is not mentioned explicitly, 
this equivalence is essentially obtained in the proof of \cite[Theorem 2.2]{OH}.

In the case of Minkowski sums of edge polytopes, 
although it seems difficult to obtain a necessary and sufficient condition to have IDP or to be normal, 
we give a sufficient condition to have IDP as follows. 

\begin{Theorem}\label{subgraph2}
Let $G_1$ be a connected graph and assume that arbitrary two odd cycles in $G_1$ always have a common vertex. 
Let $G_2$ be a subgraph of $G_1$ (not necessarily connected). 
Then $\Pc_{G_1}+\Pc_{G_2}$ has IDP, and thus, this is normal. 
\end{Theorem}
\begin{proof}
Let $\Pc=\Pc_{G_1}+\Pc_{G_2}$ and fix $\alpha \in k\Pc \cap \ZZ^d$ for a given positive integer $k$. 
Then $\alpha$ can be written like 
$$\alpha=\sum_{e \in E(G_1)}r_e \rho(e)+\sum_{e' \in E(G_2)}r_{e'}' \rho(e'),$$
where $\sum_{e \in E(G_1)}r_e=\sum_{e' \in E(G_2)}r_{e'}'=k$ and $r_e \geq 0$ (resp. $r_{e'}' \geq 0$) 
for each $e \in E(G_1)$ (resp. $e' \in E(G_2)$). 

Let $E=\{e \in E(G_1) : r_e \not\in \ZZ\}$ and $E'=\{e' \in E(G_2) : r_{e'}' \not\in \ZZ\}$. 
If $E \cap E' \not=\emptyset$, then for each $e \in E \cap E'$, 
we replace $r_e$ by $\lfloor r_e \rfloor$ and $r_e'$ by $r_e'+r_e-\lfloor r_e \rfloor$. 
Then $\alpha$ does not change and the number of edges in $E \cap E'$ decreases. 
After such replacements for all elements in $E \cap E'$, we may assume that $E \cap E' = \emptyset$. 
Then $\sum_{e \in E(G_1)}r_e$ becomes less than or equal to $k$ 
but $\sum_{e' \in E(G_2)}r_{e'}'$ becomes more than or equal to $k$.

Here one has 
\begin{align}\label{arufa}
\alpha=\underbrace{\sum_{e \in E(G_1)}\left\lfloor r_e \right\rfloor \rho(e)+
\sum_{e' \in E(G_2)} \left\lfloor r_{e'}' \right\rfloor \rho(e')}_{\in \ZZ^d}+
\underbrace{\sum_{e \in E} c_e \rho(e)+\sum_{e' \in E'} c_{e'}' \rho(e')}_{\in \ZZ^d}, 
\end{align}
where $c_e=r_e-\lfloor r_e \rfloor$ and $c_{e'}'=r_{e'}'-\lfloor r_{e'}' \rfloor$. 
Then $0 < c_e < 1$ (resp. $0 < c_{e'}' <1$) for each $e \in E$ (resp. $e' \in E'$). 
Let us consider the integer point $\beta=\sum_{e \in E} c_e \rho(e)+\sum_{e' \in E'} c_{e'}' \rho(e')$. 
Since $G_2$ is a subgraph of $G_1$, one has $E' \subset E(G_2) \subset E(G_1)$. 
Thus $\beta$ belongs to $q \Pc_{G_1} \cap \ZZ^d$, where $q=\sum_{e \in E}c_e + \sum_{e' \in E'}c_{e'}'$.

Now, Lemma \ref{key} below guarantees that $\beta$ can be written like 
$$\beta=\sum_{e \in E} a_e \rho(e) + \sum_{e \in E'} a_{e'}' \rho(e'),$$ 
where $a_j,a_j' \in \ZZ_{\geq 0},$ 
$\sum_{e \in E} a_e + \sum_{e' \in E'} a_{e'}'=q$ and $\sum_{e' \in E'} a_{e'}' \geq \sum_{e' \in E'}c_{e'}'$. 
From \eqref{arufa}, one has 
$$\alpha=
\underbrace{\sum_{e \in E(G_1)}\left\lfloor r_e \right\rfloor \rho(e)+\sum_{e \in E} a_e \rho(e)}_{\in \ZZ_{\geq 0}(\Pc_{G_1} \cap \ZZ^d)}
+\underbrace{\sum_{e' \in E(G_2)} \left\lfloor r_{e'}' \right\rfloor \rho(e')+\sum_{e' \in E'} a_{e'}' \rho(e')}_{\in \ZZ_{\geq 0}(\Pc_{G_2} \cap \ZZ^d)}.$$ 
Hence we can rewrite $\alpha$ like 
$$\alpha=\sum_{e \in E(G_1)}b_e \rho(e) + \sum_{e \in E(G_2)} b_{e'}' \rho(e'),$$ 
where $b_e\in \ZZ_{\geq 0}$ (resp. $b_{e'}' \in \ZZ_{\geq 0}$) for each $e \in E(G_1)$ (resp. $e' \in E(G_2)$), 
$\sum_{e \in E(G_1)}b_e+ \sum_{e' \in E(G_2)} b_{e'}'=2k$, 
$\sum_{e \in E(G_1)}b_e \leq k$ and $\sum_{e' \in E(G_2)}b_{e'}' \geq k$. 
Since $E' \subset E(G_1)$, we obtain an expression $\alpha$ as above satisfying 
$\sum_{e \in E(G_1)}b_e =\sum_{e' \in E(G_2)}b_{e'}' = k$. 
This means that $\alpha$ can be written as a sum of $k$ integer points in $\Pc \cap \ZZ^d$. 
Therefore, $\Pc$ has IDP, as desired. 
\end{proof}

\begin{Lemma}\label{key}
Let $G$ be a connected graph on the vertex set $\{1,\ldots,d\}$ such that 
arbitrary two odd cycles in $G$ always have a common vertex. 
Fix a positive integer $q$ and let $\alpha \in q \Pc_G \cap \ZZ^d$ having an expression 
$\alpha = \sum_{e \in E}r_e \rho(e)$, where $E \subset E(G)$ and $0 < r_e < 1$ for each $e \in E$. 
Let $E'$ be a subset of $E$ and let $q'=\sum_{e \in E'} r_e$. 
Then there exist nonnegative integers $a_e$ for $e \in E$ such that 
$\alpha=\sum_{e \in E}a_e \rho(e)$ satisfying $\sum_{e \in E'} a_e \geq q'$. 
\end{Lemma}
\begin{proof}
Given $\alpha \in q \Pc_G \cap \ZZ^d$ with an expression 
$\alpha = \sum_{e \in E}r_e \rho(e)$, where $E \subset E(G)$ and $0 < r_e < 1$ for each $e \in E$, 
let $H$ be the subgraph of $G$ whose edge set is $E$. 
Since $\alpha$ is an integer point but each $r_e$ is not an integer, 
every vertex of $H$ is always contained in at least two edges. Thus $H$ contains cycles. 

\noindent
{\bf (The first step)} 

First, we claim that $\alpha$ can be rewritten like 
$\alpha=\sum_{e \in E}a_e \rho(e)$, where $a_e \in \ZZ_{\geq 0}$ for each $e \in E$, 
by applying the following procedures (i) and (ii). 

(i) If $H$ contains an even cycle with the edges $e_1,\ldots,e_{2l}$, 
then let $\epsilon=\min\{r_{e_i} : 1 \leq i \leq 2l\}$. Without loss of generality, 
we may set $r_{e_1}=\epsilon$. We replace $r_{e_{2j-1}}$ by $r_{e_{2j-1}}-\epsilon$ 
and $r_{e_{2j}}$ by $r_{e_{2j}}+\epsilon$ for each $1 \leq j \leq l$. 
Then $\alpha$ is invariant after these replacements. 
On the other hand, the number of edges $e$ with $0 < r_e <1$ decreases at least one. 
If $e$ satisfying $1 \leq r_e < 2$ appears, then we replace $r_e$ by $r_e-1$ 
and reset $\alpha$ by $\alpha - \rho(e)$. 
We reset $H$ by the subgraph of $G$ whose edges $e \in E$ satisfy $0 < r_e <1$. 
Then such new $H$ also contains cycles. 
We repeat this procedure until $H$ contains no even cycle. 

(ii) Assume that $H$ contains no even cycle. 
Then it is easy to see that each 2-connected component of the graph is either one edge or an odd cycle. 
Thus $H$ contains at least one odd cycle. 
If there is a 2-connected component which is one edge, 
then $H$ should contain at least two odd cycles which have no common vertex, a contradiction. 
Moreover, if $H$ contains only one odd cycle, then $H$ consists of only that odd cycle. 
In this case, the sum of the entries of $\alpha$ should be odd, 
a contradiction to $\alpha \in \{(x_1,\ldots,x_d) \in \ZZ^d : \sum_{i=1}^d x_i = 2q\}$.

Hence, all 2-connected components of $H$ are odd cyles and $H$ contains at least two odd cycles. 
By our assumpstion, two odd cycles in $H$ have one common vertex and such common vertex is unique. 
Let $C$ and $C'$ be two odd cycles in $H$ having a unique common vertex $v$, 
let $v=v_1,v_2,\ldots,v_{2p+1}$ (resp. $v=v_1',v_2',\ldots,v_{2p'+1}'$) 
be vertices of $C$ (resp. $C'$) and let 
$e_i=\{v_i,v_{i+1}\}$ (resp. $e_i'=\{v_i',v_{i+1}'\}$) for $1 \leq i \leq 2p+1$ 
(resp. $1 \leq i \leq 2p'+1$), where $v_{2p+2}=v_1$ (resp. $v_{2p'+2}'=v_1'$). 
Let $\epsilon=\min\{r_{e_i}, r_{e_{i'}'}' : 1 \leq i \leq 2p+1, 1 \leq i' \leq 2p'+1\}$, say, $r_{e_1}=\epsilon$. 
We replace $r_{e_{2j-1}}$ (resp. $r_{e_{2j'}'}'$) by $r_{e_{2j-1}}-\epsilon$ (resp. $r_{e_{2j'}'}'-\epsilon$) 
for $1 \leq j \leq p+1$ (resp. $1 \leq j' \leq p'$), and 
$r_{e_{2\ell}}$ (resp. $r_{e_{2\ell'-1}'}'$) by $r_{e_{2\ell}}+\epsilon$ (resp. $r_{e_{2\ell'-1}'}'+\epsilon$) 
for $1 \leq \ell \leq p$ (resp. $1 \leq \ell' \leq p'+1$). Then $\alpha$ is invariant after these replacements 
and the number of edges $e$ with $0 < r_e <1$ decreases at least one. 
If $e$ satisfying $1 \leq r_e < 2$ appears, then we replace $r_e$ by $r_e-1$ 
and reset $\alpha$ by $\alpha - \rho(e)$. 
We reset $H$ by the subgraph of $G$ whose edges $e \in E$ satisfy $0 < r_e <1$. 
If such new $H$ also contains odd cycles, we repeat this until $H$ contains no cycle. 

Note that this algorithm terminates with finite procedures. 
After these operations (i) and (ii), we eventually obtain an expression 
\begin{align}\label{hyouji}
\alpha=\sum_{e \in E}a_e \rho(e), \text{ where } a_e \in \ZZ_{\geq 0}. 
\end{align}

\smallskip

Next, we prove that if we do the above procedures (i) and (ii) more properly, 
then we obtain a required expression of $\alpha$ for any subset $E' \subset E$ with $q'=\sum_{e \in E'}r_e$. 
In the following second and third steps, 
we prove that $\sum_{e \in E'}a_e \geq q'$ by induction on the number of the above procedures (i) and (ii). 
Assume that we obtain an expression \eqref{hyouji} with $N$ steps.

\noindent
{\bf (The second step)} 

When $N=1$, $H$ consists of one even cycle or two odd cycles having a unique common vertex. 
\begin{itemize}
\item[(i)] When $H$ is one even cycle with the edges $e_1,\ldots,e_{2l}$, 
since $N=1$, each of $r_{e_1},\ldots,r_{e_{2l}}$ should be 
$$r_{e_{2i}}=\epsilon \;\text{ and }\; r_{e_{2i-1}}=1 - \epsilon \;\text{ for }\; 1 \leq i \leq l 
\;\text{ with some }\; 0 < \epsilon < 1.$$
\item[(ii)] When $H$ consists of two odd cycles $C$ and $C'$ having a unique common vertex $v$, 
let $v=v_1,v_2,\ldots,v_{2p+1}$ (resp. $v=v_1',v_2',\ldots,v_{2p'+1}'$) 
be vertices of $C$ (resp. $C'$), let $e_i=\{v_i,v_{i+1}\}$ (resp. $e_i'=\{v_i',v_{i+1}'\}$) for $1 \leq i \leq 2p+1$ 
(resp. $1 \leq i \leq 2p'+1$), where $v_{2p+2}=v_1$ (resp. $v_{2p'+2}'=v_1'$) 
and let $\epsilon=\min\{r_{e_i}, r_{e_{i'}'}' : 1 \leq i \leq 2p+1, 1 \leq i' \leq 2p'+1\}$. 
Then each of $r_{e_1},\ldots,r_{e_{2p+1}}, r_{e_1'}',\ldots,r_{e_{2p'+1}'}'$ should be 
\begin{align*}
&r_{e_{2j}}=r_{e_{2j'-1}'}'=\epsilon \;\text{ for }\; 1 \leq j \leq p \text{ and } 1 \leq j' \leq p'+1 \\
\;\text{ and }\; &r_{e_{2j-1}}=r_{e_{2j'}'}'=1-\epsilon \;\text{ for }\; 1 \leq j \leq p+1 \text{ and } 1 \leq j' \leq p'. 
\end{align*}
\end{itemize}
In both cases, let $L_1=\{e \in E(H) : r_e=\epsilon\}$ and $L_2=\{e \in E(H) : r_e=1-\epsilon\}$. 
Let $m_1=|L_1 \cap E'|$ and $m_2=|L_2 \cap E'|$. 
If $m_1 \geq m_2$, then we replace $r_e$ by $r_e+1-\epsilon$ for each $e \in L_1$ 
and we also replace $r_{e'}$ by $r_{e'}-1+\epsilon$ for each $e' \in L_2$. 
Thus we obtain 
$$q'=\sum_{e \in E'}r_e=\epsilon m_1 + (1-\epsilon)m_2=m_2+\epsilon(m_1-m_2) \leq m_1$$ 
by $0 < \epsilon < 1$ and $m_1 \geq m_2$. If $m_2 \geq m_1$, after similar replacements, we obtain 
$q'=\epsilon m_1 + (1-\epsilon)m_2=m_1+(1-\epsilon)(m_2-m_1) \leq m_2.$ 
These mean that $\sum_{e \in E'} a_e=\max\{m_1,m_2\} \geq q'$. 

\noindent
{\bf (The third step)} 

Assume $N > 1$. We do the procedure (i) or (ii) as in the first step. 
\begin{itemize}
\item[(i)] When $H$ contains an even cycle with the edges $e_1,\ldots,e_{2l}$, 
let $L_1=\{ e_{2i-1} : 1 \leq i \leq l\}$ and $L_2=\{ e_{2i} : 1 \leq i \leq l\}$. 
\item[(ii)] When $H$ contains no even cycle, 
there are two odd cycles $C$ and $C'$ having a unique common vertex $v$. 
Work with the same notation as in the second step. 
Let $L_1=\{e_{2i-1} : 1 \leq i \leq p+1\} \cup \{e_{2i'} : 1 \leq i' \leq p'\}$ and 
$L_2=\{e_{2i} : 1 \leq i \leq p\} \cup \{e_{2i'-1} : 1 \leq i' \leq p'+1\}$. 
\end{itemize}
In both cases, let $m_1=|L_1 \cap E'|$ and $m_2=|L_2 \cap E'|$. 
Assume that $m_1 \geq m_2$. (The case $m_1 \leq m_2$ can be discussed by the same manner.) 
Then we set $\epsilon = \min\{r_{e'} : e' \in L_2\}$ and we replace $r_e$ by $r_e+\epsilon$ for each $e \in L_1$ 
and $r_{e'}$ by $r_{e'}-\epsilon$ for each $e' \in L_2$. 
After these replacements, if there is $r_e$ with $r_e \geq 1$ (but $r_e<2$), 
then we reset $r_e$ by $r_e-1$. Let 
$E''=\{e \in E' : r_e \text{ becomes } r_e \geq 1 \text{ after the replacements}\}$. 
Then $q'=\sum_{e \in E'}r_e$ changes into $q'+\epsilon(m_1-m_2)-|E''|$ after the replacements. 
By the inductive hypothesis, there exist $a_e$'s such that 
$\alpha - \sum_{e \in E''}\rho(e)=\sum_{e \in E'} a_e \rho(e)$ 
with $a_e \in \ZZ_{\geq 0}$ and $\sum_{e \in E'}a_e \geq q'+\epsilon(m_1-m_2)-|E''|$. 
Hence, we obtain $$|E''|+\sum_{e \in E'} a_e \geq q' + \epsilon(m_1-m_2) \geq q'.$$ 
Since $\alpha=\sum_{e \in E''}\rho(e)+\sum_{e \in E'} a_e\rho(e)$ and 
$|E''|+\sum_{e \in E'}a_e \geq q'$, we obtain the required assertion. 
\end{proof}

\begin{Remark}
The condition ``arbitrary two odd cycles in $G_1$ always have a common vertex'' in Theorem \ref{subgraph2} is 
stronger condition than the odd cycle condition. 
\end{Remark}

\smallskip

\subsection{Examples}
Finally, we conclude this paper by the following examples, 
which show that each condition described in Theorem \ref{subgraph2} is necessary.

\begin{Examples}
(a) The following example shows that 
the assumption ``two odd cycles always have a common vertex'' is necessary. 
Let $G_1$ and $G_2$ be graphs in Figure \ref{counter2}. Then $G_2$ is a subgraph of $G_1$. We see that 
\begin{figure}[htb!]
\centering
\includegraphics[scale=0.33]{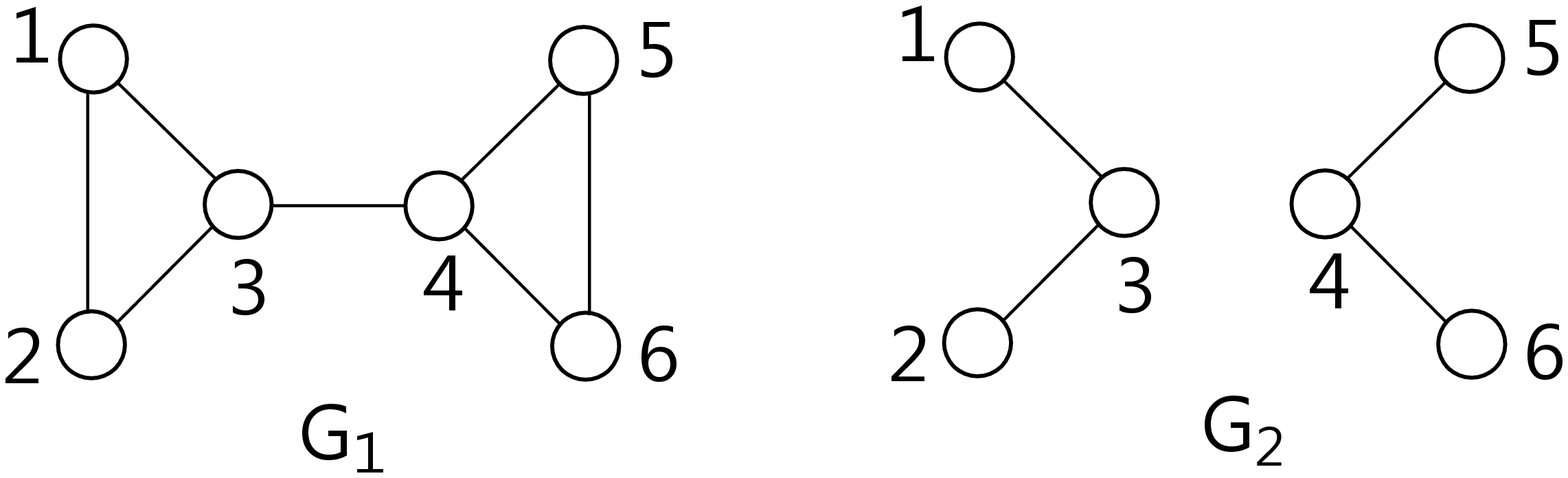}
\caption{An example showing that our assumption is necessary}\label{counter2}
\end{figure}
\begin{align*}
\left(\frac{1}{2}\rho(\{1,2\})+\frac{3}{2}\rho(\{5,6\})\right) + 
\left(\frac{1}{2}\rho(\{1,3\})+\frac{1}{2}\rho(\{2,3\}) + \frac{1}{2}\rho(\{4,5\})+\frac{1}{2}\rho(\{4,6\})\right) \\
=(1,1,1,1,2,2) \in 2\Pc \cap \ZZ^6 = 2 \Pc \cap L(2\Pc), 
\end{align*}
where $\Pc = \Pc_{G_1}+\Pc_{G_2}$. 
Since $(1,1,1,1,2,2)$ cannot be written as any sum of two integer points in $\Pc \cap \ZZ^6$, 
$\Pc$ is not normal. In particular, $\Pc$ does not have IDP. 

(b) Next, the following example shows that 
the assumption ``$G_2$ is a subgraph of $G_1$'' is also necessary. 
Let $G_1$ and $G_2$ be graphs in Figure \ref{counter3}. 
Then each of $G_1$ and $G_2$ satisfies that two odd cycles always have a common vertex. We also see that 
\begin{figure}[htb!]
\centering
\includegraphics[scale=0.33]{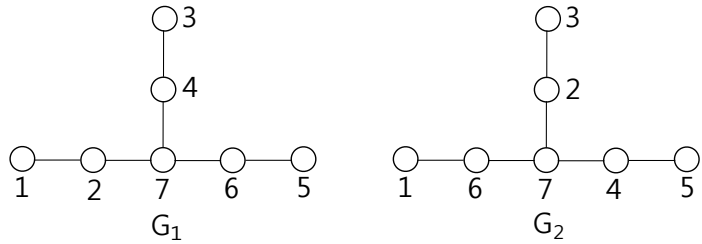}
\caption{Other example showing that our assumption is necessary}\label{counter3}
\end{figure}
\begin{align*}
\left(\frac{4}{3}\rho(\{1,2\})+\frac{1}{3}\rho(\{3,4\})+\frac{1}{3}\rho(\{5,6\})\right) + 
\left(\frac{2}{3}\rho(\{1,6\})+\frac{2}{3}\rho(\{2,3\})+\frac{2}{3}\rho(\{4,5\})\right) \\
=(2,2,1,1,1,1) \in 2\Pc \cap \ZZ^6 = 2\Pc \cap L(2\Pc), 
\end{align*}
where $\Pc=\Pc_{G_1}+\Pc_{G_2}$. 
Since $(2,2,1,1,1,1)$ cannot be written as a sum of any two integer points in $\Pc \cap \ZZ^6$, 
$\Pc$ is not normal, and thus, this does not have IDP. 

(c) In addition, Theorem \ref{subgraph2} is no longer true for the case of three graphs. 
Let $G_1$, $G_2$ and $G_3$ be graphs in Figure \ref{counter}. 
Then $G_2$ is a subgraph of $G_1$ and so is $G_3$ and $G_3$ is a subgraph of $G_2$, too. 
Note that $G_1$ satisfies that two odd cycles have a common vertex. We also have 
\begin{figure}[htb!]
\centering
\includegraphics[scale=0.32]{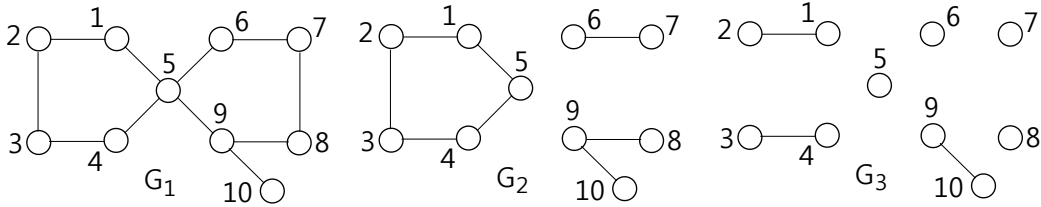}
\caption{A counterexample for Theorem \ref{subgraph2} in the case of three graphs}\label{counter}
\end{figure}
\begin{align*}
&\quad\left(\frac{3}{5}\rho(\{5,6\})+\frac{3}{5}\rho(\{7,8\}) + \frac{3}{5}\rho(\{5,9\})+\frac{1}{5}\rho(\{9,10\})\right) \\
&+\left(\frac{2}{5}\rho(\{1,5\})+\frac{2}{5}\rho(\{2,3\}) + \frac{2}{5}\rho(\{4,5\})+\frac{2}{5}\rho(\{6,7\})+\frac{2}{5}\rho(\{8,9\})\right) \\
&+\left(\frac{3}{5}\rho(\{1,2\})+\frac{3}{5}\rho(\{3,4\}) + \frac{4}{5}\rho(\{9,10\})\right) \\
&=(1,1,1,1,2,1,1,1,2,1) \in 2\Pc \cap \ZZ^{10}= 2\Pc \cap L(2\Pc), 
\end{align*}
where $\Pc=\Pc_{G_1}+\Pc_{G_2}+\Pc_{G_3}$. 
One can check that $(1,1,1,1,2,1,1,1,2,1)$ cannot be written as any sum of two integer points in $\Pc \cap \ZZ^{10}$. 
Thus this is not normal. In particular, this does not have IDP. 
\end{Examples}


\smallskip

\begin{thebibliography}{99}


\bibitem{BrGu}
W. Bruns and J. Gubeladze, 
``Polytopes, rings and K-theory'', 
Springer--Verlag, Heidelberg, 2009. 

\bibitem{BGT}
W. Bruns, J. Gubeladze and N. V. Trung, 
Normal polytopes, triangulations, and Koszul algebras, 
{\em J. Reine Angew. Math.} {\bf 485} (1997), 123--160. 

\bibitem{CLS}
D. Cox, J. Little and H. Schenck, 
``Toric varieties,'' American Mathematical Society, 2011. 

\bibitem{workshop}
C. Haase, T. Hibi and D. Maclagan (organizers), 
Mini-Workshop: Projective Normality ofSmooth Toric Varieties, 
OWR {\bf 4} (2007), 2283--2320. 


\bibitem{OH}
H. Ohsugi and T. Hibi, 
Normal polytopes arising from finite graphs, {\em J. Algebra} {\bf 207} (1998), 409--426. 


\bibitem{Sch}
A. Schrijver, ``Theory of Linear and Integer Programming,'' John Wiley \& Sons, 1986. 

\bibitem{SVV}
A. Simis, W. V. Vasconcelos and R. H. Villarreal, 
The integral closure of subrings associated to graphs, {\em J. Algebra} {\bf 199} (1998), 281--289. 

\end{thebibliography}
\end{document}